\documentclass[letterpaper, 10 pt, conference]{ieeeconf} 
\IEEEoverridecommandlockouts
\usepackage{cite}
\usepackage{amsmath,amssymb,amsfonts}
\usepackage{graphicx}
\usepackage[hidelinks]{hyperref}
\usepackage{textcomp}
\usepackage{mathrsfs}
\usepackage{mathtools}
\usepackage{stmaryrd}
\usepackage{booktabs}
\usepackage{siunitx}
\usepackage{physics}

\mathtoolsset{showonlyrefs}

\newtheorem{definition}{Definition}
\newtheorem{assumption}{Assumption}

\newtheorem{theorem}{Theorem}

\newtheorem{problem}{Problem}

\newcommand{\asmref}[1]{Assumption~\ref{#1}}

\newcommand{\probref}[1]{Problem~\ref{#1}}

\newcommand{\secref}[1]{Section~\ref{#1}}

\newcommand{\figref}[1]{Fig.~\ref{#1}}
\newcommand{\tabref}[1]{Table~\ref{#1}}

\allowdisplaybreaks[4]

\title{\LARGE \bf
Equivalence Between Continuous-Time Risk-Sensitive Control and R{\'e}nyi Divergence Minimization
}

\author{Shinji Kataoka, Kaoru Teranishi, and Yasumasa Fujisaki
\thanks{This work was not supported by any organization}
\thanks{S. Kataoka, K. Teranishi, and Y. Fujisaki are with the Department of Information and Physical Sciences, Graduate School of Information Science and Technology, The University of Osaka, 1-5 Yamadaoka, Suita, Osaka 565-0871, Japan {\tt\small \{shinji-kataoka, k-teranishi, fujisaki\}@ist.osaka-u.ac.jp}.}%
}

\begin{document}

\maketitle
\thispagestyle{empty}
\pagestyle{empty}

\begin{abstract}
In this study, we show that a continuous-time risk-sensitive control problem is equivalent to a R{\'e}nyi divergence minimization problem over trajectory path measures.
Reformulating stochastic optimal control as probabilistic inference via Kullback-Leibler (KL) divergence minimization avoids the computational intractability of the Hamilton-Jacobi-Bellman equation.
However, standard KL control is inherently risk-neutral, and recent minimax extensions remain restricted to risk-averse settings.
Our equivalence result resolves this limitation by offering a unified probabilistic framework for arbitrary risk attitudes in continuous-time nonlinear systems.
Based on Girsanov theorem, we explicitly map the risk sensitivity to the R{\'e}nyi divergence order, deriving a noise-dependent control penalty scaled by risk preference.
This formulation seamlessly modulates tail-weighting behaviors, interpolating between zero-forcing for risk-averse policies and mass-covering for risk-seeking policies.
These findings bridge stochastic control and information-theoretic inference, providing a foundation for sampling-based control algorithms. 
\end{abstract}

\section{Introduction}
\label{sec:introduction}

Optimal control of nonlinear stochastic systems presents a fundamental challenge in managing complex dynamical processes.
Because physical systems operating in real-world environments are subject to stochastic disturbances, designing control laws that minimize the expected cost over uncertain system trajectories is essential to guarantee both stability and performance~\cite{fleming_deterministic_1975}.
Theoretically, such optimal control laws are characterized by the Hamilton-Jacobi-Bellman equation.
However, this nonlinear partial differential equation is often analytically intractable.
To overcome this limitation, alternative frameworks reformulate stochastic optimal control as sampling-based control~\cite{kappen_linear_2005,Todorov2009-on,theodorou_relative_2012}.

These frameworks leverage information-theoretic metrics to measure the deviation of the controlled system from a desired reference behavior.
For instance, Kullback-Leibler (KL) control employs the KL divergence between the probability measure over controlled trajectories and that over reference trajectories~\cite{Todorov2009-on,theodorou_relative_2012,theodorou_nonlinear_2015}.
KL control admits a path integral representation of the optimal control input, thereby enabling its computation via Monte Carlo sampling~\cite{kazim_recent_2024}.
Under certain structural assumptions, this formulation is equivalent to conventional stochastic optimal control.

Although KL control enables efficient computation of optimal control inputs, it is inherently risk-neutral and does not explicitly account for extreme risks or tail events.
To address this limitation, extending this framework to risk-sensitive control~\cite{Jacobson1973-iy,fleming_risk_2006,Fleming2005-pw} has attracted significant interest for handling unexpected disturbances and worst-case scenarios.
For example, safety-critical applications such as autonomous driving and medical robotics require risk-averse control to avoid high-cost failures under severe disturbances~\cite{ahmadi_risk_2022,chapman_risk-sensitive_2022}.
Conversely, tasks such as exploration in unknown environments or financial decision-making may benefit from risk-seeking behavior, which accepts potential risks to achieve higher returns~\cite{nass_entropic_2019,sarfi_risk-sensitive_2025,berkelaar_optimal_2004}.
Therefore, an information-theoretic control framework that can flexibly accommodate different risk sensitivities based on task requirements is desirable.

A recent study introduced a minimax KL control framework that extends KL control to a two-player zero-sum game setting~\cite{patil_path_2025}.
In this setup, an adversarial player attempts to maximize the cost while the controller minimizes it under worst-case disturbances.
Although the authors demonstrated that this minimax control is equivalent to risk-averse control, their framework inherently cannot accommodate risk-seeking scenarios.
Another promising approach is to employ the R{\'e}nyi divergence instead of the KL divergence.
As a generalization of the KL divergence, the R{\'e}nyi divergence incorporates a tunable parameter to adjust the emphasis placed on different regions of the probability distribution~\cite{van_erven_renyi_2014}.
In machine learning, this flexibility has been leveraged in variational inference~\cite{li_renyi_2016}.
Furthermore, the R{\'e}nyi divergence has been applied to the control-as-inference framework, establishing the equivalence between R{\'e}nyi-divergence-based control as inference and discrete-time risk-sensitive control~\cite{ito_risk-sensitive_2024}.
However, this framework has not yet been extended to continuous-time systems.

To bridge this gap, the primary objective of this study is to integrate R{\'e}nyi variational inference into stochastic optimal control.
Specifically, for continuous-time nonlinear stochastic systems, we establish the equivalence between the risk-sensitive control problem and the R{\'e}nyi divergence minimization problem.
By exploiting the properties of the R{\'e}nyi divergence, our proposed framework incorporates a tunable parameter that seamlessly accommodates both risk-averse and risk-seeking policies within a unified probabilistic formulation.

The main contributions of this study are summarized as follows:
\begin{itemize}
    \item We establish an equivalence between continuous-time risk-sensitive control and probabilistic inference using R{\'e}nyi divergence.
    \item We provide a unified information-theoretic framework for both risk-averse and risk-seeking control, overcoming the structural restriction to risk-averse settings present in existing minimax KL control.
    \item We provide a control-theoretic perspective on R{\'e}nyi variational inference, establishing a direct mapping between geometric approximation behaviors and control risk attitudes.
\end{itemize}

The remainder of this paper is organized as follows.
\secref{sec:problem_formulation} introduces a continuous-time nonlinear stochastic system and formulates both the risk-sensitive control and R{\'e}nyi divergence minimization problems.
\secref{sec:main_result} presents the main results, providing the proof of equivalence between the two formulations.
Finally, \secref{sec:conclusion} concludes the paper and discusses directions for future work.

\section{Problem formulation}
\label{sec:problem_formulation}

\subsection{Notation}

Let $\mathbb{R}$ denote the set of real numbers.
For a vector $v \in \mathbb{R}^n$, its Euclidean norm is denoted by $\norm{v}$.
For a matrix $M \in \mathbb{R}^{m \times n}$, its pseudoinverse is denoted by $M^+$.
Let $(\Omega, \mathcal{F})$ be a measurable space.
For probability measures $P$ and $Q$ on $(\Omega, \mathcal{F})$, we say that $P$ is absolutely continuous with respect to $Q$ if, for every $A \in \mathcal{F}$, $Q(A) = 0$ implies $P(A) = 0$.
Furthermore, $P$ and $Q$ are said to be equivalent if they are mutually absolutely continuous.
If $P$ is absolutely continuous with respect to $Q$, then by the Radon-Nikodym theorem, there exists an $\mathcal{F}$-measurable non-negative function $Z$ such that $dP = Z dQ$.
This function $Z$ is called the Radon-Nikodym derivative of $P$ with respect to $Q$ and is denoted by $dP / dQ$.
The expectation of a random variable $X$ under $P$ is denoted by $\mathbb{E}^P[X]$.
The KL divergence of $P$ with respect to $Q$ is defined as $D(P \parallel Q) \coloneqq \mathbb{E}^P [ \log dP / dQ ]$ if the Radon-Nikodym derivative $dP / dQ$ exists, and $D(P \parallel Q) \coloneqq +\infty$ otherwise.

\subsection{System Description and Assumptions}

Consider a continuous-time nonlinear stochastic system described by
\begin{equation}
    dx_t = f(t, x_t) dt + g(t, x_t)u_t dt + h(t, x_t) dw_t,
    \label{eq:dynamics}
\end{equation}
where $x_0=x$, $x_t \in \mathbb{R}^n$ is the state, $u_t \in \mathbb{R}^m$ is the control input, $w_t \in \mathbb{R}^r$ is a standard Brownian motion, $f_t: \mathbb{R}^n \to \mathbb{R}^n$, $g_t: \mathbb{R}^n \to \mathbb{R}^{n \times m}$, and $h_t: \mathbb{R}^n \to \mathbb{R}^{n \times r}$.
Throughout this paper, we assume the existence and uniqueness of a solution to \eqref{eq:dynamics} over the finite time interval $[0, T]$.
We denote the state trajectory of \eqref{eq:dynamics} over $[0, T]$ by $x_{0:T}$ and its corresponding path measure by $P$.
For system \eqref{eq:dynamics}, the uncontrolled dynamics (i.e., with $u_t = 0$) is given by
\begin{equation}
    dx_t = f(t, x_t) dt + h(t, x_t) dw_t,
    \label{eq:dynamics_0}
\end{equation}
where $x_0=x$.
We denote the path measure of the state trajectory for this system by $Q$.

We next state two assumptions on system \eqref{eq:dynamics} to ensure the validity of the change of measure between the controlled and uncontrolled path measures.
The first assumption introduces structural conditions on the drift and diffusion coefficients of the system.

\begin{assumption}\label{ass:col_space}
    For all $t \in [0, T]$ and $x \in \mathbb{R}^n$, the following conditions hold:
    \begin{itemize}
        \item The column spaces of the input matrix $g(t, x)$ and the diffusion matrix $h(t, x)$ are identical.
        \item The input matrix $g(t, x)$ has full column rank.
    \end{itemize}
\end{assumption}

It follows from this assumption that for all $t \in [0, T]$ and $x \in \mathbb{R}^n$, given any control vector $u \in \mathbb{R}^m$, there exists a corresponding vector $\gamma \in \mathbb{R}^r$ satisfying $g(t, x) u = h(t, x) \gamma$.
In particular, one such vector $\gamma$ is explicitly given by $\gamma = h(t, x)^+ g(t, x) u$ because $h(t, x) \gamma = (h(t, x) h(t, x)^+) g(t, x) u = g(t, x) u$, where $h(t, x) h(t, x)^+$ is the orthogonal projection onto the column space of $h(t, x)$.
In practice, this implies that the noise and the control input act on the system dynamics through identical channels.
This property guarantees that any control input can be represented as a drift in the underlying noise process.

With the control input formulated as a drift $\gamma_t = h(t, x_t)^+ g(t, x_t) u_t$, the validity of the measure change requires the corresponding exponential local martingale to be a true martingale.
In stochastic calculus, this requirement is ensured by an exponential integrability condition known as Novikov's condition~\cite{Oksendal1995-hs}.
The second assumption introduces this condition for the drift term $\gamma_t$, thereby justifying the Radon-Nikodym derivative.

\begin{assumption}
\label{ass:input}
    For system \eqref{eq:dynamics}, the control input $u_t$ satisfies, for every $k \in (0, \infty)$,
    \begin{equation}
        \mathbb{E}^Q \qty[ \exp( \frac{k}{2} \int_0^T u_t^{\top} g_t^{\top} (h_t h_t^{\top})^+ g_t u_t dt ) ] < \infty,
        \label{eq:input_condition}
    \end{equation}
    where $g_t = g(t, x_t)$ and $h_t = h(t, x_t)$.
\end{assumption}

Under this integrability condition, Girsanov theorem~\cite{Oksendal1995-hs} guarantees a valid measure change.
In other words, the controlled path measure $P$ is equivalent to the uncontrolled path measure $Q$.
From a control perspective, this measure equivalence excludes unrealistic control inputs that would inject infinite energy into the system over a finite time horizon.
If an excessive control input violates this condition, absolute continuity between the path measures is lost, rendering the probabilistic evaluation via Girsanov theorem inapplicable.

\subsection{Risk-sensitive control}

Building on the system dynamics defined above, this section introduces the risk-sensitive control problem~\cite{Jacobson1973-iy,fleming_risk_2006}.
For system \eqref{eq:dynamics}, let $c(t, x_t, u_t) \ge 0$ denote the running cost at time $t \in [0, T]$.
In a standard optimal control problem, we minimize the expected cumulative cost.
This risk-neutral approach evaluates costs linearly and often underestimates rare, extreme trajectories.
In contrast, risk-sensitive control minimizes the logarithm of the expected exponential cumulative cost, formulated as follows.

\begin{problem}[Risk-Sensitive Control]
\label{prob:risk_sensitive}
    \begin{alignat}{2}
        &\min_{P} &\quad &\frac{1}{\theta} \log \mathbb{E}^{P} \qty[ \exp( \theta \int_0^T c(t, x_t, u_t) dt ) ] \label{eq:risk_sensitive_objective} \\
        &\,\mathrm{s.t.} &\quad &\eqref{eq:dynamics}, \ \theta > -1, \ \theta \neq 0
    \end{alignat}
\end{problem}

Note that the risk-sensitive formulation recovers the standard risk-neutral optimal control problem
\begin{alignat}{2}
    &\min_{P} &\quad &\mathbb{E}^{P} \qty[ \int_0^T c(t, x_t, u_t) dt ] \\
    &\,\text{s.t.} &\quad &\eqref{eq:dynamics}
\end{alignat}
in the limit as $\theta \to 0$.
The logarithmic and exponential transformations introduce a sensitivity to higher-order moments of the cumulative cost, establishing an explicit trade-off between its expectation and variance.
Applying a Taylor expansion with respect to $\theta$ around zero yields
\begin{align}
    &\frac{1}{\theta} \log \mathbb{E}^{P} \qty[ \exp( \theta \int_0^T c(t, x_t, u_t) dt ) ] \\
    &\approx \mathbb{E}^{P} \qty[ \int_0^T c(t, x_t, u_t) dt ] + \frac{\theta}{2} \mathbb{V}^{P} \qty[ \int_0^T c(t, x_t, u_t) dt ],
\end{align}
where $\mathbb{V}^{P}$ denotes the variance under path measure $P$.
Here, the parameter $\theta$ determines the degree of risk sensitivity.
When $\theta > 0$, the objective penalizes the cost variance, resulting in risk-averse behavior that prioritizes suppressing high-cost tail events.
Conversely, when $\theta < 0$, the system exhibits risk-seeking behavior that tolerates worst-case scenarios in pursuit of lower expected costs.

\subsection{R{\'e}nyi divergence minimization}

The goal of this study is to establish an equivalence between the risk-sensitive control problem in \probref{prob:risk_sensitive} and a specific variational inference problem.
To this end, this section introduces the R{\'e}nyi divergence and formulates a minimization problem based on the divergence of the controlled path measure from a target path measure.
The R{\'e}nyi divergence, a generalization of the KL divergence, is an information-theoretic quantity defined as a pseudo-distance between two probability measures.

\begin{definition}[R{\'e}nyi divergence~\cite{van_erven_renyi_2014}]
\label{def:renyi_divergence}
    Let $P$ and $Q$ be probability measures such that $P$ is absolutely continuous with respect to $Q$.
    The R{\'e}nyi divergence of order $\alpha$ of $P$ with respect to $Q$ is defined as
    \begin{equation}
        D_\alpha(P \parallel Q) \coloneqq \frac{1}{\alpha - 1} \log \mathbb{E}^P \qty[ \qty( \frac{dP}{dQ} )^{\alpha - 1} ],
        \label{eq:def_renyi_divergence}
    \end{equation}
    where $\alpha > 0$ and $\alpha \neq 1$.
\end{definition}

Note that in the limit as $\alpha \to 1$, the R{\'e}nyi divergence converges to the KL divergence~\cite{van_erven_renyi_2014}, namely
\begin{equation}
    \lim_{\alpha \to 1} D_\alpha(P \parallel Q) = D(P \parallel Q).
\end{equation}
The order parameter $\alpha$ adjusts the sensitivity to the distribution tails when quantifying the dissimilarity between probability measures.
\figref{fig:renyi_comparison} depicts three normal distributions: $\mathcal{N}(0, 1)$ (blue), $\mathcal{N}(2, 1)$ (red), and $\mathcal{N}(3, 1)$ (orange).
Let $P$, $Q$, and $R$ be the Gaussian measures corresponding to these distributions.
\tabref{tab:renyi_values} then lists the R{\'e}nyi divergence values of $P$ relative to $Q$ and $R$ for $\alpha \in \{0.1, 0.5, 1.0, 2.0, 3.0\}$.

Comparing $Q$ and $R$, we observe that $D_\alpha(P \parallel Q) < D_\alpha(P \parallel R)$ holds for all evaluated $\alpha$.
Furthermore, the difference between $D_\alpha(P \parallel R)$ and $D_\alpha(P \parallel Q)$ increases as $\alpha$ increases.
This indicates that a larger spatial separation between distributions yields a higher divergence, and this penalty is magnified by larger values of $\alpha$.

To understand how the order $\alpha$ determines which area of the distribution is emphasized, we focus on $D_\alpha(P \parallel Q)$.
The divergence calculation evaluates the expectation $\mathbb{E}^P [ ( dP / dQ )^{\alpha - 1} ]$, where the Radon-Nikodym derivative term weights different spatial regions under the measure $P$.
For these specific Gaussian measures, setting $\alpha = 0.5$ emphasizes the overlapping region between $P$ and $Q$.
As $\alpha$ increases, the emphasized region shifts toward the left tail of $P$, where the measure $Q$ lacks probability mass.
Consequently, increasing $\alpha$ heightens the sensitivity to low-probability tail deviations.

\begin{figure}[t]
    \centering
    \includegraphics[scale=1]{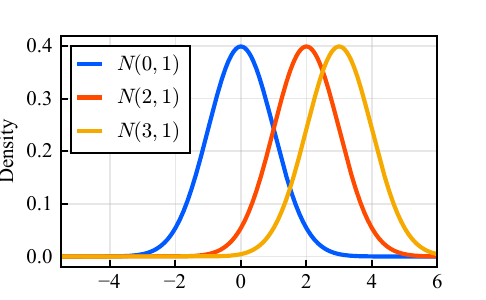}
    \caption{Normal distributions with different means and variances.}
    \label{fig:renyi_comparison}
\end{figure}

\begin{table}[t]
    \caption{R{\'e}nyi divergences for the Gaussian measures.}
    \label{tab:renyi_values}
    \centering
    \begin{tabular}{S[table-format=1.1] S[table-format=1.2] S[table-format=2.2]}
        \toprule
        {$\alpha$} & {$D_\alpha(P \parallel Q)$} & {$D_\alpha(P \parallel R)$} \\
        \midrule
        0.1 & 0.20 & 0.45 \\
        0.5 & 1.00 & 2.25 \\
        1.0 & 2.00 & 4.50 \\
        2.0 & 4.00 & 9.00 \\
        3.0 & 6.00 & 13.50 \\
        \bottomrule
    \end{tabular}
\end{table}

Using the R{\'e}nyi divergence, we formulate a minimization problem.
In what follows, we assume that the running cost is given by
\begin{equation}
    c(t, x_t, u_t) \coloneqq \ell(t, x_t) + \frac{1}{2} u_t^{\top} R_t u_t,
    \label{eq:running_cost}
\end{equation}
where the first term $\ell: [0, T] \times \mathbb{R}^n \to \mathbb{R}$ is a nonnegative state cost, and the second term is a quadratic control cost with a positive-definite weight matrix $R_t \in \mathbb{R}^{m \times m}$.
Moreover, we define the total state cost over the time interval $[0, T]$ as
\begin{equation}
    C(x_{0:T}) \coloneqq \int_0^T \ell(t, x_t) dt.
    \label{eq:total_cost}
\end{equation}

Using this total state cost, we define the reference path measure $Q^\ast$ by
\begin{equation}
    dQ^\ast \coloneqq \frac{1}{Z} \exp(-\frac{1}{\lambda}C(x_{0:T})) dQ, \quad \lambda > 0,
    \label{eq:def_measure_Q_star}
\end{equation}
where $Q$ is the path measure for the system \eqref{eq:dynamics_0}, and $Z = \mathbb{E}^Q[\exp(-(1/\lambda)C(x_{0:T}))]$ is a normalization constant.
Since the total state cost $C(x_{0:T})$ is given by the integral of the nonnegative cost $\ell(t, x_t)$, the cost satisfies $C(x_{0:T}) \ge 0$, which guarantees that the constant $Z$ is finite.
Consequently, the reference measure $Q^\ast$ is equivalent to the uncontrolled path measure $Q$, and their Radon-Nikodym derivative is given by
\begin{equation}
    \frac{dQ}{dQ^\ast} = Z \exp(\frac{1}{\lambda}C(x_{0:T})).
\end{equation}

In the context of KL control, the reference path measure $Q^\ast$ is interpreted as the minimizer of the free energy associated with the total state cost $C(x_{0:T})$ relative to the uncontrolled path measure $Q$~\cite{theodorou_relative_2012}.
Furthermore, the normalization constant $Z$ corresponds to the partition function in this statistical mechanics analogy.
Under this framework, the total state cost $C(x_{0:T})$ plays the role of energy, making lower-cost trajectories exponentially more probable.

Having established the reference path measure $Q^\ast$, we now formally define the R{\'e}nyi divergence minimization problem considered in this study.

\begin{problem}[R{\'e}nyi divergence minimization]
\label{prob:renyi}
    \begin{alignat}{2}
        &\min_{P} &\quad &D_\alpha(P \parallel Q^\ast) \label{eq:renyi_objective} \\
        &\,\mathrm{s.t.} &\quad &\eqref{eq:dynamics}, \ \eqref{eq:dynamics_0}, \ \eqref{eq:def_measure_Q_star}, \ \alpha > 0, \ \alpha \neq 1
    \end{alignat}
\end{problem}

This formulation seeks the optimal controlled path measure $P$ for the system \eqref{eq:dynamics} that minimizes its R{\'e}nyi divergence relative to the reference path measure $Q^\ast$.
Specifically, attaining $D_\alpha(P \parallel Q^\ast) = 0$ implies that the controlled path measure coincides with the target measure $Q^\ast$.
Here, the order parameter $\alpha$ controls the tail weighting relative to the target measure $Q^\ast$, and thus its value is directly related to the risk sensitivity $\theta$ in \probref{prob:risk_sensitive}.
When $\alpha > 1$, the formulation yields risk-averse control, whereas $0 < \alpha < 1$ yields risk-seeking control.
The following section formally shows the equivalence between \probref{prob:risk_sensitive} and \probref{prob:renyi}, and clarifies the relationship between the parameters $\alpha$ and $\theta$.

\section{Main Results}
\label{sec:main_result}

In this section, we establish the equivalence between the risk-sensitive control problem and the R{\'e}nyi divergence minimization problem.
Our main result is formally presented in the following theorem.

\begin{theorem}
    Consider the risk-sensitive control problem in \probref{prob:risk_sensitive} and the R{\'e}nyi divergence minimization problem in \probref{prob:renyi}.
    Under \asmref{ass:col_space} and \asmref{ass:input}, if
    \begin{equation}
        R_t = \frac{\lambda}{\alpha} g_t^{\top} \qty( h_t h_t^{\top} )^+ g_t \quad \text{and} \quad \theta = \frac{\alpha - 1}{\lambda}
    \end{equation}
    hold, then the two problems are equivalent, where $g_t = g(t, x_t)$ and $h_t = h(t, x_t)$.
\end{theorem}

\begin{proof}
    Note that \asmref{ass:input} guarantees that $R_t$ is positive definitie for all $t \in [0, T]$.
    In what follows, we write $f_t = f(t, x_t)$, $g_t = g(t, x_t)$, $h_t = h(t, x_t)$, and $C = C(x_{0:T})$ to simplify the notation.
    Furthermore, to explicitly indicate that the Brownian motions in \eqref{eq:dynamics} and \eqref{eq:dynamics_0} are defined under their respective measures $P$ and $Q$, we denote them by $w_t^P$ and $w_t^Q$.

    First, we evaluate the Radon-Nikodym derivative $dP/dQ^\ast$ associated with the objective function in \eqref{eq:renyi_objective}.
    Since both path measures $P$ and $Q^\ast$ are absolutely continuous with respect to the uncontrolled path measure $Q$, we can apply the chain rule for Radon-Nikodym derivatives,
    \begin{equation}
        \frac{dP}{dQ^\ast} = \frac{dP}{dQ} \frac{dQ}{dQ^\ast}.
    \end{equation}
    From the definition of the reference measure $Q^\ast$ in \eqref{eq:def_measure_Q_star}, we have
    \begin{equation}
        \frac{dQ}{dQ^\ast} = Z \exp(\frac{1}{\lambda}C).
    \end{equation}
    Substituting this into the chain rule formula yields
    \begin{equation}
        \frac{dP}{dQ^\ast} = Z \exp(\frac{1}{\lambda}C) \frac{dP}{dQ} \propto \exp(\frac{1}{\lambda}C) \frac{dP}{dQ}.
        \label{eq:dP_dQ_star}
    \end{equation}
    Because the normalization constant $Z$ is independent of the control input $u_t$, it is inherently independent of the controlled path measure $P$.
    This implies that $Z$ does not affect the minimization in \probref{prob:renyi} and can be omitted from the optimization problem.

    Next, we characterize the Radon-Nikodym derivative $dP/dQ$ in \eqref{eq:dP_dQ_star} by leveraging the system dynamics.
    Under the uncontrolled path measure $Q$, the system trajectory is governed by the standard Brownian motion $w_t^Q$, whereas the control input $u_t$ introduces the drift term $g_t u_t dt$.
    To connect the controlled path measure $P$ with $Q$, we project the drift $g_t u_t dt$ into the stochastic noise space.
    From \asmref{ass:col_space}, we can choose the scaled drift process
    \begin{equation}
        \gamma_t \coloneqq h_t^+ g_t u_t
        \label{eq:gamma_t}
    \end{equation}
    such that $h_t \gamma_t = g_t u_t$.
    Then, under \asmref{ass:input}, the exponential process defined by
    \begin{equation}
        M_t \coloneqq \exp( \int_0^t \gamma_s^{\top} dw_s^Q - \frac{1}{2} \int_0^t \norm{ \gamma_s }^2 ds ), \quad t \in [0, T]
        \label{eq:def_M_t}
    \end{equation}
    becomes a true martingale under $Q$, ensuring $\mathbb{E}^Q[M_T] = 1$.
    By Girsanov theorem, the Radon-Nikodym derivative of $P$ with respect to $Q$ is therefore given by
    \begin{equation}
        \frac{dP}{dQ} = M_T.
    \end{equation}
    Substituting this result into \eqref{eq:dP_dQ_star} yields
    \begin{equation}
        \qty( \frac{dP}{dQ^\ast} )^{\alpha-1} \propto M_T^{\alpha-1} \exp( \frac{\alpha-1}{\lambda} C ).
        \label{eq:dP_dQ_alpha_1}
    \end{equation}
    
    Again, by Girsanov theorem, the standard Brownian motion under $Q$ can be expressed as $dw_t^Q = \gamma_t dt + dw_t^P$.
    Substituting this into $M_T^{\alpha-1}$ in \eqref{eq:dP_dQ_alpha_1} yields
    \begin{align}
        &M_T^{\alpha-1} = \exp( \int_0^T \gamma_s^{\top} dw_s^P + \frac{1}{2} \int_0^T \norm{ \gamma_s }^2 ds )^{\alpha-1}, \\
        &= \exp( (\alpha-1) \int_0^T \gamma_s^{\top} dw_s^P + \frac{\alpha-1}{2} \int_0^T \norm{ \gamma_s }^2 ds ).
    \end{align}
    To simplify the exponent of $M_T^{\alpha-1}$, we define the modified drift process
    \begin{equation}
        \tilde{\gamma}_t \coloneqq (\alpha - 1) \gamma_t,
    \end{equation}
    along with its corresponding exponential process under $P$:
    \begin{equation}
        \tilde{M}_t \coloneqq \exp( \int_0^t \tilde{\gamma}_s^{\top} dw_s^P - \frac{1}{2} \int_0^t \norm{ \tilde{\gamma}_s }^2 ds ), \quad t \in [0, T].
    \end{equation}
    Completing the square in the exponent to isolate $\tilde{M}_T$, we obtain
    \begin{equation}
        M_T^{\alpha-1} = \tilde{M}_T \exp( \frac{\alpha(\alpha-1)}{2} \int_0^T \norm{ \gamma_s }^2 ds ),
        \label{eq:M_T_alpha_1}
    \end{equation}
    where
    \begin{align}
        &(\alpha-1) \int_0^T \gamma_s^{\top} dw_s^P + \frac{\alpha-1}{2} \int_0^T \norm{ \gamma_s }^2 ds  \\
        &= \int_0^T \tilde{\gamma}_s^{\top} dw_s^P - \frac{(\alpha-1)^2}{2} \int_0^T \norm{ \gamma_s }^2 ds, \\
        &\quad + \frac{(\alpha-1)^2}{2} \int_0^T \norm{ \gamma_s }^2 ds + \frac{\alpha-1}{2} \int_0^T \norm{ \gamma_s }^2 ds, \\
        &= \int_0^T \tilde{\gamma}_s^{\top} dw_s^P - \frac{1}{2} \int_0^T \norm{ \tilde{\gamma}_s }^2 ds + \frac{\alpha(\alpha-1)}{2} \int_0^T \norm{ \gamma_s }^2 ds,
    \end{align}
    and the first two terms of the last equality correspond to the exponent of $\tilde{M}_T$.

    We now introduce a new path measure $\tilde{P}$ defined by $d\tilde{P} \coloneqq \tilde{M}_T dP$.
    To determine the system dynamics under $\tilde{P}$, we observe from Girsanov theorem that the standard Brownian motion under this new measure is governed by $dw_t^{\tilde{P}} = -\tilde{\gamma}_t dt + dw_t^P$.
    Using this relation, the original system dynamics \eqref{eq:dynamics} can be rewritten under $\tilde{P}$ as
    \begin{align}
        dx_t
        &= f_t dt + g_t u_t dt + h_t dw_t^{\tilde{P}} + h_t \tilde{\gamma}_t dt, \\
        &= f_t dt + g_t u_t dt + h_t dw_t^{\tilde{P}} + (\alpha - 1) g_t u_t dt, \\
        &= f_t dt + g_t \tilde{u}_t dt + h_t dw_t^{\tilde{P}}, \label{eq:dynamics_tilde_P}
    \end{align}
    where $x_0 = x$ and $\tilde{u}_t = \alpha u_t$.
    Under the scaled control input $\tilde{u}_t$, the squared norm of the process $\gamma_t$ in \eqref{eq:gamma_t} satisfies
    \begin{equation}
        \norm{ \gamma_t }^2 = \frac{1}{\alpha^2} \tilde{u}_t^{\top} g_t^{\top}(h_th_t^{\top})^+ g_t \tilde{u}_t = \frac{1}{\alpha\lambda} \tilde{u}_t^{\top} R_t \tilde{u}_t.
        \label{eq:norm_gamma_t}
    \end{equation}
    Furthermore, from \asmref{ass:input}, because the process $\tilde{M}_t$ is a true martingale under $P$, expectations under $\tilde{P}$ and $P$ are related by $\mathbb{E}^P [\tilde{M}_T (\cdot)] = \mathbb{E}^{\tilde{P}} [\cdot]$.
    Applying this identity alongside \eqref{eq:M_T_alpha_1}, we express the expectation of \eqref{eq:dP_dQ_alpha_1} as
    \begin{align}
        &\mathbb{E}^P \qty[ \qty( \frac{dP}{dQ^\ast} )^{\alpha-1} ] \\
        &\propto \mathbb{E}^P \qty[ M_T^{\alpha-1} \exp( \frac{\alpha-1}{\lambda}C ) ], \\
        &= \mathbb{E}^{\tilde{P}} \qty[ \exp( \frac{\alpha-1}{\lambda}C + \frac{\alpha(\alpha-1)}{2} \int_0^T \norm{ \gamma_s }^2 ds ) ], \\
        &= \mathbb{E}^{\tilde{P}} \qty[ \exp( \frac{\alpha-1}{\lambda}C + \frac{\alpha-1}{2\lambda} \int_0^T \tilde{u}_s^{\top} R_s \tilde{u}_s ds ) ], \\
        &=  \mathbb{E}^{\tilde{P}} \qty[ \exp{ \theta \qty( C + \frac{1}{2} \int_0^T \tilde{u}_s^{\top} R_s \tilde{u}_s ds ) } ]. \label{eq:Expectation_P_tildeP}
    \end{align}
    Therefore, from \eqref{eq:running_cost} and \eqref{eq:total_cost}, the objective function in \eqref{eq:renyi_objective} satisfies
    \begin{align}
        &D_\alpha(P \parallel Q^\ast) \\
        &\propto \frac{1}{\theta} \log \mathbb{E}^{\tilde{P}} \qty[ \exp{ \theta \qty( C + \frac{1}{2} \int_0^T \tilde{u}_t^{\top} R_t \tilde{u}_t dt ) } ], \\
        &= \frac{1}{\theta} \log \mathbb{E}^{\tilde{P}} \qty[ \exp( \theta \int_0^T c(t, x_t, \tilde{u}_t) dt ) ].
    \end{align}
    Since the final expression coincides with the risk-sensitive control objective in \eqref{eq:risk_sensitive_objective}, the equivalence of the two problems is established.
    This completes the proof.
\end{proof}

This theorem establishes the equivalence between \probref{prob:risk_sensitive} and \probref{prob:renyi}.
Given a risk-sensitive control problem satisfying \asmref{ass:col_space} and \asmref{ass:input}, if the matching condition,
\begin{equation}
    h(t, x_t) h(t, x_t)^\top = \frac{\lambda}{1 + \lambda \theta} \, g(t, x_t) R_t^{-1} g(t, x_t)^\top,
\end{equation}
holds for all $t \in [0, T]$ with some $\lambda > 0$ satisfying $1 + \lambda \theta > 0$, then the corresponding R{\'e}nyi divergence minimization problem can be constructed by choosing the order $\alpha$ as
\begin{equation}
    \alpha = 1 + \lambda \theta.
\end{equation}
Moreover, the optimal control input $\tilde{u}_t^\ast$ for \probref{prob:risk_sensitive} can be recovered from the optimal control input $u_t^\ast$ that induces the optimal solution to \probref{prob:renyi} via the relation $\tilde{u}_t^\ast = \alpha u_t^\ast$.

As the order $\alpha \to 1$, the R{\'e}nyi divergence $D_\alpha(P \parallel Q^\ast)$ converges to the KL divergence $D(P \parallel Q^\ast)$.
In this limit, R{\'e}nyi divergence minimization recovers KL control.
Indeed, using \eqref{eq:total_cost} and \eqref{eq:dP_dQ_star}, it follows that
\begin{align}
    D(P \parallel Q^\ast)
    &= \mathbb{E}^P \qty[ \log \frac{dP}{dQ^\ast} ], \\
    &= \mathbb{E}^P \qty[ \log \frac{dP}{dQ} + \frac{1}{\lambda} C(x_{0:T}) ] + \log Z, \\
    &\propto \mathbb{E}^P \qty[ \int_0^T \ell(t, x_t) dt ] + D(P \parallel Q).
\end{align}
Correspondingly, the risk sensitivity $\theta \to 0$ reduces the risk-sensitive control problem to a risk-neutral one.
Note that, via Girsanov theorem, KL control is well known to be equivalent to risk-neutral control~\cite{kazim_recent_2024,patil_path_2025}.

The control-cost weighting matrix $R_t$ in the theorem represents a noise-dependent control penalty.
Specifically, it penalizes control inputs more heavily in directions with smaller noise variance, thereby favoring control action along directions in which the system is more strongly excited by the noise.
Furthermore, for a fixed $\lambda$, this matrix is scaled by $1 / \alpha$, directly coupling the control cost with the risk parameter.
In a risk-averse setting ($\alpha > 1$), the control penalty $R_t$ is attenuated, enabling the controller to apply larger control inputs to actively suppress state deviations and avoid worst-case scenarios.
Conversely, in a risk-seeking setting ($0 < \alpha < 1$), the control penalty $R_t$ is amplified, penalizing control effort more severely and leading the system to tolerate disturbances rather than expending control effort to reject them.

From the perspective of R{\'e}nyi variational inference~\cite{li_renyi_2016}, the divergence order $\alpha$ modulates the trade-off between \emph{zero-forcing} and \emph{mass-covering} behaviors.
A zero-forcing approximation discourages the approximating distribution from assigning probability mass to regions of low target density, whereas a mass-covering approximation tends to cover regions carrying significant probability mass under the target distribution.
Through the equivalence established in the theorem, these approximation tendencies translate directly to distinct risk attitudes in control.
In the risk-averse setting ($\alpha > 1$, i.e., $\theta > 0$), minimizing $D_\alpha(P \parallel Q^\ast)$ strengthens the zero-forcing tendency.
Because $Q^\ast$ assigns exponentially small probability to high-cost trajectories, the controlled path measure $P$ is strongly discouraged from allocating substantial mass to such trajectories, thereby steering the system away from high-cost regions.
Conversely, in the risk-seeking setting ($0 < \alpha < 1$, i.e., $\theta < 0$), the approximation becomes relatively more mass-covering.
Consequently, placing probability mass in regions where $Q^\ast$ is small is penalized less severely, allowing the system to tolerate high-cost tail events rather than expending excessive control effort.
Unlike the existing minimax KL control framework~\cite{patil_path_2025}, which is restricted to risk-averse settings ($\theta > 0$), our approach provides a unified information-theoretic formulation for both risk-averse and risk-seeking control via the intrinsic properties of R{\'e}nyi divergence.

Interestingly, the authors of~\cite{li_renyi_2016} observed that $\alpha = 0.5$ provided a favorable empirical trade-off between test log-likelihood and prediction error in their Bayesian neural-network regression experiments.
At this order, the R{\'e}nyi divergence satisfies
\begin{equation}
    D_{0.5}(P \parallel Q) = -2 \log( 1 - \frac{\mathrm{Hel}^2(P, Q)}{2} ),
\end{equation}
where $\mathrm{Hel}^2(P, Q)$ denotes the squared Hellinger distance~\cite{van_erven_renyi_2014}.
Additionally, $\alpha = 0.5$ is the unique order at which the R{\'e}nyi divergence is symmetric.
This symmetry may provide a useful compromise between zero-forcing and mass-covering tendencies.
In our control formulation, $\alpha = 0.5$ corresponds to the risk sensitivity $\theta = -1 / (2\lambda) < 0$, representing a specific risk-seeking control policy.
While whether this choice provides an optimal trade-off in practice depends on the specific task, it serves as a meaningful baseline for risk-seeking scenarios.

\section{Conclusion}
\label{sec:conclusion}

In this study, we established the equivalence between the risk-sensitive control problem and the R{\'e}nyi divergence minimization problem for continuous-time nonlinear stochastic systems.
By leveraging Girsanov theorem and constructing an appropriate change of path measure, we derived the corresponding control cost matrix $R_t$ and the explicit relationship between the parameters $\theta$ and $\alpha$.
This equivalence bridges information-theoretic control formulations and risk-sensitive optimal control.
Moreover, it provides a unified probabilistic inference framework for both risk-averse and risk-seeking control problems.

Future work will focus on developing solution methods for the R{\'e}nyi divergence minimization problem.
We also plan to clarify the relationships between the R{\'e}nyi divergence framework and other optimal control formulations.

\bibliographystyle{IEEEtran}
\bibliography{reference}

\end{document}